\documentclass[10pt]{article}
\usepackage[utf8]{inputenc} 

\usepackage{sectsty}
\usepackage{graphicx}       
\usepackage{amsmath}        
\usepackage{amsfonts}
\usepackage{amsthm}
\usepackage{amssymb}
\usepackage[affil-it]{authblk}
\usepackage{color}

\newtheorem{thm}{Theorem}[section]

\newtheorem{lemma}[thm]{Lemma}
\newtheorem{proposition}[thm]{Proposition}
\newtheorem{assumption}[thm]{Assumption}

\theoremstyle{definition}

\usepackage{breakcites}     
\usepackage[margin=1in]{geometry}   
\usepackage{sidecap}        
\usepackage{setspace}       
\usepackage[numbers]{natbib}

\usepackage{hyperref}       
\hypersetup{
    colorlinks,
    citecolor=black,
    filecolor=black,
    linkcolor=black,
    urlcolor=black
}

\allsectionsfont{\singlespacing}

\begin{document}

\title{Multilevel Markov Chain Monte Carlo for Bayesian inverse problems for Navier Stokes equation with Lagrangian Observations}

\author{Yang Juntao%
  \thanks{Email address: \texttt{yjuntao@nvidia.com}}}
\affil{ Nvidia AI Technology Center, Singapore \\
Division of Mathematical Sciences, School of Physical and Mathematical Sciences, \\
Nanyang Technological University, Singapore}

\date{}
\maketitle

\section{Introduction}
We consider the Bayesian inverse problem for recovering the initial velocity and the random forcing of forward Navier-Stokes equation with noisy observation on the position of tracers at some time moments. Data inversion for Navier-Stokes equation is important for areas like weather forecasting, ocean modeling and aerospace engineering. For weather forecasting and ocean modeling, it is common to have positional data from tracers such as weather balloons and drifters. However, sampling the posterior probability of Bayesian inverse problem for the forward Navier-Stokes equation coupled with tracer equation is highly expensive if not intractable. The Markov Chain Monte Carlo (MCMC) sampling procedure requires a large number of realizations of the forward equation to obtain a reasonable level of accuracy, hence leads to high complexity.

For linear elliptic forward problems, in the case of the uniform prior probability measure, Hoang, Schwab and Stuart \cite{ Hoang2013} develop the Multilevel Markov Chain Monte Carlo (MLMCMC) method that approximates the expectation with respect to the posterior probability measure of quantities of interest by solving the forward equation by finite elements (FEs) with different levels of resolution. The number of MCMC  samples is chosen judiciously according to the FE resolution level. The method is essentially optimal. To obtain an approximation for the posterior expectation of a quantity of interest within a prescribed level of accuracy, the total number of degrees of freedom required for solving all the realizations of the forward equation  in the sampling process is equivalent to that for solving only one realization of the forward equation to obtain an equivalent level of FE accuracy. Comparing to the plain MCMC procedure where a large number of realizations of the forward equation is solved with an equally high level of accuracy, the computation time is drastically reduced. The convergence of the method is rigorously proved. Numerical experiments verify the theoretical convergence rate. We have early applied MLMCMC method to Bayesian inverse problem for inferring the unknown forcing and initial condition of the forward Navier-Stokes equation with noisy Eulerian observations on the velocity. \cite{juntao2022}. Rigorous theory has been developed for the case of uniform prior where the forcing and the initial condition depend linearly on a countable set of random variables which are uniformly distributed in a compact interval. 

In this paper, we extend our work to the Bayesian inverse problems for inferring unknown forcing and initial condition of the forward Navier-Stokes equation coupled with tracer equation with noisy Lagrangian observation on the positions of the tracers. We consider the Navier-Stokes equations in the two dimensional periodic torus with a tracer equation which is a simple ordinary differential equation. We developed rigorously the theory for the case of the uniform prior where the forcing and the initial condition depend linearly on a countable set of random variables which are uniformly distributed in a compact interval. Numerical experiment using the MLMCMC method produces approximations for posterior expectation of quantities of interest which are in agreement with the theoretical optimal convergence rate established. 

\section{Parametric Navier Stokes equation}
Let $\mathbb{T}^2$ be the two dimensional unit torus in $\mathbb{R}^2$. Let $[0, T]$ be the time interval of interest. We denote the following function spaces $V:= [H^1(\mathbb{T}^2)]^2$, $M = \{p\in L^2(\mathbb{T}^2), \int_{\mathbb{T}^2} p = 0\}$, $\mathbf{H}=\{u\in [L^2(\mathbb{T}^2)]^2, \: \nabla \cdot u = 0\} $ and $\mathbf{H}^s = \{u \in [H^s(\mathbb{T}^2)]^2, \: \nabla \cdot u = 0\}$ with $\|\cdot\|$ the norm on $\mathbf{H}$ and $\|\cdot\|_s$ the norm on $\mathbf{H}^s$.     We denote by $\mathcal{H} = L^2(0, T; \mathbf{H})$ and $\mathcal{H}^s = L^2(0, T; \mathbf{H}^s)$ with $|\cdot|_0$ the norm in $\mathcal{H}$ and $|\cdot|_s$ the norm in $\mathcal{H}^s$ We consider the two dimensional Navier-Stokes equation in $\mathbb{T}^2$ 
\begin{equation}
\label{navier stokes equation}
    \begin{cases}
    \frac{\partial u }{\partial t} + u \cdot \nabla u - \nu \Delta u + \nabla p = f, \quad & \forall (x, t) \in \mathbb{T}^2 \times (0, T], \\
    \nabla \cdot u = 0, \quad & \forall (x, t) \in \mathbb{T}^2 \times (0, T], \\
    u = u_0, \quad &\forall x \in \mathbb{T}^2 \text{ and } t = 0,
    \end{cases}
\end{equation}
where $u$ and $p$ are the velocity and the pressure, $\nu > 0$ is the viscosity. Let $v$ be a test function in space $V$ and $r$ be a scalar test function in space $M$.  
We denote by $(\cdot, \cdot)$  the inner product in the space $L^2(\mathbb{T}^2)^2$, extended by density to the duality pairing between $V'$ and $V$. We have the following mixed weak formulation of the Navier-Stokes equation,
\begin{equation}
\label{weak form}
    \begin{cases}
    \text{Find $(u(t), p(t)) \in V \times M$ such that} \\
    \frac{d}{d t}(u, v)+(u \cdot \nabla u, v)+(\nu \nabla u, \nabla v)-(p, \nabla \cdot v) =(f, v), &\quad \forall v \in V, \\
    (\nabla \cdot u, r) =0, &\quad \forall r \in M,\\
    u(0, \cdot) = u_0.
    \end{cases}
\end{equation}
The existence and uniqueness of a solution for the above two dimensional problem is ensured by theorem 3.1 and theorem 3.2 in \cite{Temam1978}.
If $u_0 \in \mathbf{H}$ and $f \in L^2(0, T; V')$, problem \eqref{weak form} possesses a solution $u$ which  belongs to $L^2(0, T; V)$.
However, in this paper we need more regularity for the solution to set up and establish the well-posedness of the Bayesian inverse problem. We thus assume more regularity for the forcing $f$ and the initial condition $u_0$. In particular, we assume a divergence free initial velocity $u_0 \in \mathbf{H}^1$ and $f\in L^2(0, T; \mathbf{H})$ so that $u \in C(0, T; V)$ (see lemma 5.2 in \cite{Cotter2009}).
 By lemma 5.2 in \cite{Cotter2009} we also have $u \in C(0, T; \mathbf{H}^1) \cap L^2(0, T; \mathbf{H}^2)$ which satisfies
\begin{equation}
\label{boundedness}
\|u(t)\|_1^2 + \int_0^t \|u(\tau)\|_2^2 d \tau \leq C(\|u_0\|_1, |f|_0),
\end{equation}
where $C(\|u_0\|_1, |f|_0)$ is a constant. The Bayesian inverse problem is to infer the unknown initial condition and random forcing given observations on the velocity  at a set of times $\{\tau_k\}_{k=1}^K$. Furthermore, to obtain FE error estimates for the Navier-Stokes equation, we need more regularity assumptions on $u_0$ and $f$ (see sections 1-3 of \cite{He2008} for example). Thus in section \ref{sect:FEapproximate}, we will assume more regularity for $u_0$ and $f$. Now we consider $J$ tracers $\{z_j\}_{j=1}^J$ which are transported by the velocity field $u$ of the Navier-Stokes equation \eqref{navier stokes equation}. The trajectories of the tracers are governed by the following ODE equation,
\begin{equation}
\label{transport equation}
\frac{d z_j}{d t} = u(z_j, t), \qquad z_j(0) = z_{j,0},
\end{equation}
where $\{z_{j,0}\}_{j=1}^J$ is a set of initial positions. 
We assume that the initial velocity $u_0$ and the forcing $f$ are represented in the parametric form as
\begin{equation}
\label{series expansion}
\begin{aligned}
    u_0(x) &= \Bar{u}_0(x) + \sum_{i \geq 1} \zeta_i \phi_i(x), \quad & x \in \mathbb{T}^2,\\
    f(t,x) &= \Bar{f}(t,x) + \sum_{i \geq 1} \xi_i \psi_i(t,x), \quad & x \in \mathbb{T}^2, \; t \in (0, T],
\end{aligned}    
\end{equation}
where  $\zeta_i$ and $\xi_i$ are normalized such that $\zeta_i, \xi_i \in [-1, 1] \text{ for } i=1,2,...$ and $\phi_i \in \mathbf{H}^1$, $\psi_i \in L^2(0, T; \mathbf{H})$. We define the sequences $\zeta = (\zeta_1, \zeta_2, ...) \in [-1, 1]^{\mathbb{N}}$ and $\xi = (\xi_1, \xi_2, ...)\in [-1, 1]^{\mathbb{N}}$. To indicate the dependence of $u_0$ on $\zeta$ and of $f$ on $\xi$, and of $u$ on $\zeta, \xi$, we write them as $u_0(x, \zeta)$, $f(t, x, \xi)$, and $u(t, x, \zeta, \xi)$. \\
We make the following assumptions on the decay rate of the sequences $\{\phi_i\}$ and $\{\psi_i\}$. 
\begin{assumption}
\label{random field}
The functions $\Bar{u}_0$, $\phi_i$ are in $\mathbf{H}^1$ and $\Bar{f}$, $\psi_i$ are in $\mathcal{H}$. There exist constants $s > 1$ and $C > 0$ such that
    \[\forall i \in \mathbb{N}: \quad \|\phi_i\|_1 \leq C i^{-s}, \quad |\psi_i|_0 \leq C i^{-s}.\]
\end{assumption}
\noindent
With assumption \ref{random field}, there exist finite positive constants ${u_0}_{\max}$ and $f_{\max}$ such that  
\begin{align*}
&\|u_0(\zeta)\|_1 \leq \|\bar{u} _0\|_1+\sum_{j=1}^{\infty}\|\phi_i\|_1 \leq {u_0}_{\max},\\
& |f(\xi)|_0 \leq |\Bar{f}|_0+ \sum_{j=1}^{\infty}|\psi_i|_0 \leq f_{\max}.
\end{align*}
From assumption \ref{random field}, we deduce that there is a constant $C > 0$ such that
for all $I\in \mathbb{N}$,
\begin{equation}
\label{truncation estimate}
\sum_{i>I} \|\phi_i\|_1 < C I^{-q}, \quad \sum_{i>I} |\psi_i|_0 < C I^{-q}, 
\end{equation}
where $q = s-1$. 
Next, we describe the prior measure. 
We denote the probability space $U = [-1, 1]^{\mathbb{N}} \bigotimes [-1, 1]^{\mathbb{N}}$, the set of all pairs $(\zeta, \xi)$ of sequences $\zeta = \{\zeta_i\}_{i\geq 1}$ of coordinates $\zeta_i$ and $\xi = \{\xi_i\}_{i\geq 1}$ of coordinates $\xi_i$, s.t. $\zeta_i \in [-1, 1]$ and $\xi_i \in [-1, 1] \text{ for all } i$. 
We denote the product sigma algebra on the parameter domain $U$ by $\Theta = \left(\bigotimes_{i=1}^{\infty}\mathcal{B}([-1, 1])\right) \bigotimes \left(\bigotimes_{i=1}^{\infty}\mathcal{B}([-1, 1])\right)$, where $\mathcal{B}$ is the Borel sigma algebra. Due to assumption \ref{random field}, the series in (\ref{series expansion}) converge in $\mathbf{H}^1$ and $\mathcal{H}$. Assuming that $\zeta_i$ and $\xi_i$ are uniformly distributed in $[-1, 1]$, we define the product probability on the measurable space $(U, \Theta)$ by
\begin{equation}
\label{eqn probability space}
\gamma = \left(\bigotimes_{i=1}^{\infty} \frac{d \zeta_i}{2}\right) \bigotimes \left(\bigotimes_{i=1}^{\infty}\frac{d \xi_i}{2}\right), 
\end{equation}
where $d\zeta_i$ and $d\xi_i$ denote the Lebesgue measure in $\mathbb{R}$.

\section{Bayesian inverse problem}
In this section, we present the setting for the Bayesian inverse problem. We consider general observations of the trajectories of drifting tracers in the two dimensional velocity field governed by the two dimensional Navier-Stokes equation. The trajectories $\{z_j\}_{j=1}^J$ are observed at a set of times $\{ \tau_k \}_{k=1}^K \subset (0, T]$. We define the forward observation map $\mathcal{G}: U \rightarrow [\mathbb{R}^2]^{JK} $ for all $(\zeta, \xi) \in U $ as 
\begin{equation}
    \label{eqn:forward map}
    \mathcal{G}(\zeta, \xi) := \{z_j(\tau_k)\}_{j, k=1}^{J, K}.
\end{equation}
Let $\vartheta = \{\sigma_{j,k}\}_{j,k=1}^{J,K}$ be the observation noise. $\sigma_{j,k} \in \mathbb{R}^2$ is assumed Gaussian and independent of the parameters $\zeta$ and $\xi$. Thus the random variable $\vartheta$ has values in $[\mathbb{R}^2]^{JK}$ and follows the normal distribution $N(0, \Sigma)$, where $\Sigma$ is a known $2JK \times 2JK$ symmetric positive covariance matrix. The noisy observation $y$ is 
\begin{equation*}
    y = \mathcal{G}(\zeta, \xi) + \vartheta. 
\end{equation*}
The posterior probability measure is the conditional probability of $(\zeta, \xi)$ in $U$ given observation $y$. We define the mismatch function 
\begin{equation}
\label{eqn:mismatch function}
    \Phi(\zeta, \xi; y) = \frac{1}{2}\|y-\mathcal{G}(\zeta, \xi)\|_{\Sigma}^2,
\end{equation}
where $\|\cdot\|_{\Sigma}^2 = \langle \Sigma^{-1/2}\cdot, \Sigma^{-1/2} \cdot \rangle$ with $\langle \cdot, \cdot \rangle$ being the inner product in $[\mathbb{R}^2]^{JK}$. 
\begin{lemma}
Under assumption \ref{random field}, the forward map $\mathcal{G}(\zeta, \xi): U \rightarrow [\mathbb{R}^2]^{JK}$ is continuous as a mapping from the measurable space $(U, \Theta)$ to $([\mathbb{R}^2]^{JK}, \mathcal{B}([\mathbb{R}^2]^{JK}))$.
\end{lemma}
\begin{proof}
Following from Lemma 3.13 from \cite{Cotter2009}, we have 
\begin{align*}
    |\mathcal{G}(\zeta, \xi)-\mathcal{G}(\zeta',\xi')| &\leq C(\|u_0(\zeta)-v_0(\zeta')\|_1 + |f(\xi)-f(\xi')|_0) \\
    &\leq C\left(\sum_{i\geq 1} |\zeta_i-\zeta_i'|\|\phi_i\|_1+\sum_{i\geq 1}|\xi_i-\xi_i'||\psi_i|_0\right) \\
    &\leq C(\|\zeta-\zeta'\|_{l^\infty(\mathbb{N})}+\|\xi-\xi'\|_{l^\infty(\mathbb{N}))}.    
\end{align*}
Hence we have the continuity of $\mathcal{G}$ with respect to $\zeta, \xi$.
\end{proof}
With the continuity with respect to the random parameters proven, we have the following theorem on the existence of Randon-Nikodym derivative as a result of corollary 2.2. in \cite{Cotter2009}. 
\begin{thm}
The posterior probability measure $\gamma^\delta$ is absolutely continuous with respect to the prior $\gamma$. The Radon-Nikodym derivative is given by 
\begin{equation}
    \frac{d \gamma^y}{d \gamma} \propto \exp(-\Phi(\zeta, \xi; y))
\end{equation}
\end{thm}
Next we consider the continuity of the posterior measure in the Hellinger distance with respect to the observation data, which implies the well-posedness of the posterior measure. The Hellinger distance is defined as 
\begin{equation}
    d_{Hell}(\gamma', \gamma'') = \left(\frac{1}{2}\int_U(\sqrt{\frac{d\gamma'}{d\gamma}}-\sqrt{\frac{d\gamma''}{d\gamma}})d\gamma\right)^{1/2},
\end{equation}
where $\gamma'$ and $\gamma''$ are two measures on $U$, which are absolutely continuous with respect to the measure $\lambda$. It is shown in \cite{Hoang2013} that the Lipschitzness of the posterior measure with respect to the Hellinger distance holds under general conditions. 
\begin{proposition}
The measure $\gamma^y$ depends locally Lipschitz continuously on the data $y$ with respect to the Hellinger metric: for every $r > 0$ and $y, y' \in \mathbb{R}^d$ such that for $|y|_\Sigma, |y'|_\Sigma \leq r$, there exists $C = C(r) > 0$ such that 
\begin{equation}
d_{Hell}(\gamma^y, \gamma^{y'}) \leq C(r)|y-y'|_\Sigma. 
\end{equation}
\end{proposition}
The proof of the proposition is similar to the one for Proposition 3.3. in \cite{juntao2022}. 

\section{Posterior approximation by finite truncation of the forcing and the initial condition}
We consider the approximation of the forward equation by truncating the series expansion \eqref{series expansion} for the forcing and the initial velocity after $I$ terms. Let 
\begin{equation}
\label{eqn:truncated expansion}
    u_0^I(x, \zeta) = \bar{u}_0(x) + \sum_{i=1}^I \phi_i(x)\zeta_i, \quad f^I(t, x, \xi) = \bar{f}(t, x) + \sum_{i=1}^I \psi_i(t, x) \xi_i. 
\end{equation}
We consider the truncated problem, 
\begin{equation}
\label{truncated weak form}
    \begin{cases}
    \text{Find } u^I(t) \in V, p^I(t) \in M \text{ such that }\\
    \frac{d}{d t}(u^I, v)+(u^I \cdot \nabla u^I, v)+(\nu \nabla u^I, \nabla v)-(p^I, \nabla \cdot v) =( f^I, v), &\quad \forall v \in V,\\
    (\nabla \cdot u^I, r) =0, &\quad \forall r \in M,\\
    u^I(0, \cdot) = u_0^I,
    \end{cases}
\end{equation}
with the tracer trajectory 
\begin{equation}
    \frac{d z_j^I}{d t} = u(z_j^I, t), \qquad z_j^I(0) = z_{j,0}.
\end{equation} 
\begin{proposition}
\label{eqn:forward map truncation error}
Under assumption \ref{random field}, the truncated forward map $\mathcal{G}^I(\zeta, \xi) := \{z_j^I(\tau_k)\}_{j,k=1}^{J,K}$ satisfies the estimate,
\begin{equation}
    |\mathcal{G}(\zeta, \xi)-\mathcal{G}^I(\zeta, \xi)| \leq C I^{-q}. 
\end{equation}
\end{proposition}
\begin{proof}
With $u_0 \in \mathbf{H}^1$ and $f \in \mathcal{H}^1$, we have
\begin{equation*}
|\mathcal{G}(\zeta, \xi)-\mathcal{G}^I(\zeta, \xi)| \leq C(\|u_0\|_1, \|u_0^I\|_1, |f|_1, |f^I|_1)(\|u_0-u_0^I\|_s^2 + |f-f^I|_0^2)^{\frac{1}{2}}.
\end{equation*}
The preceding conclusion is by Lemma 3.13 in \cite{Cotter2009}. Due to assumption \ref{random field} and $\|u_0 - u_0^I\|_1^2 \leq C I^{-2q}, |f-f^I|_0^2 \leq C I^{-2q}$. Hence, we conclude that 
\begin{equation*}
|\mathcal{G}(\zeta, \xi)-\mathcal{G}^I(\zeta, \xi)| \leq C I^{-q},
\end{equation*}
where $C$ is a constant uniformly bounded with respect to $(\zeta, \xi) \in U$.
\end{proof}
Then we define the approximated posterior measure $\gamma^{I,y}$ as,
\begin{equation*}
    \frac{d \gamma^{I,y}}{d \gamma} \propto \exp(-\Phi^I(\zeta, \xi; y)), 
\end{equation*}
where $\Phi^I(\zeta, \xi; y)$ is the potential function
\begin{equation}
\Phi^I(\zeta, \xi; y) = \frac{1}{2}|y - \mathcal{G}^I(\zeta, \xi)|_\Sigma^2. 
\end{equation}
The measure $\gamma^{I, y}$ is an approximation of the Bayesian posterior. Next we show the error estimate for the approximation of the posterior measure by the solution of the truncated equation in the Hellinger metric. 
\begin{proposition}
Under assumption \ref{random field} and proposition \ref{eqn:forward map truncation error}, there is a constant $C(y) > 0$ such that for every $I$
\begin{equation}
    d_{Hell}(\gamma^y, \gamma^{I, y}) \leq C(y) I^{-q}.
\end{equation}
\end{proposition}

\section{FE Approximation of the truncated problem} 
\label{sect:FEapproximate}
We describe the FE approximation of solution $u^I, p^I$ of \eqref{truncated weak form} with the truncated forcing and initial condition in \eqref{eqn:truncated expansion}. In the two dimensional periodic unit torus $\mathbb{T}^2$, we define the following nested family $\{\mathcal{T}^l\}_{l=0}^{\infty}$ of simplicial partition of $\mathbb{T}^2$. 
The domain $\mathbb{T}^2$ is first subdivided into a regular family $\mathcal{T}^0$ of simplices $T$ which are periodically distributed; then for $l \geq 1$, each simplex in $\mathcal{T}^l$ is obtained by subdividing each simplex in $\mathcal{T}^{l-1}$ into $4$ congruent triangles. Hence the mesh size $h_l = \max\{\text{diam($T$)}: T \in \mathcal{T}^l\}$ of $\mathcal{T}^l$ is $h_l=2^{-l}h_0$. We define the following nested hierarchical family of spaces of $\mathbb{P}_1$-iso-$\mathbb{P}_2/\mathbb{P}_1$ FE spaces on $\mathcal{T}^l$ as 
\begin{align*}
V^l &= \{u \in V : u|_T \in [\mathcal{P}_1(T)]^{2} \quad \forall T \in \mathcal{T}^{l+1} \}, \\
M^l &= \{p \in M : p|_T \in \mathcal{P}_1(T) \quad \forall T \in \mathcal{T}^l \},    
\end{align*}
where $\mathcal{P}_1(T) $ denotes the set of linear functions in the simplex $T \in \mathcal{T}^l \text{ or } T \in \mathcal{T}^{l+1}$. With the FE approximation space defined, we consider the FE approximation of the truncated problem
\begin{equation}
\label{fem approximation}
\begin{cases}
    \text{Seek } u^{I,l}(t) \in V^l \text{ and } p^{I,l}(t) \in M^l \text{ such that} \\
    (\frac{du^{I,l}}{dt}, v^l) + (\nu \nabla u^{I,l}, \nabla v^l) + ((u^{I,l} \cdot \nabla) u^{I,l}, v^l) + (\nabla p, v^l) = (f^I, v^l), &\forall v^l \in V^l, \\
    (\nabla \cdot u^{I,l}, r^l) = 0,  &\forall r^l \in M^l, \\
\end{cases}    
\end{equation}
with $I,l \in \mathbb{N}$. We choose simplicial partition with $\mathbb{P}_1$-iso-$\mathbb{P}_2/\mathbb{P}_1$ FE spaces in the numerical analysis, but the analysis is also valid for nested hyperrectangular partition with $\mathbb{Q}_1$-iso-$\mathbb{Q}_2/\mathbb{Q}_1$ FE pair. The wellposedness of the two dimensional problem for approximating the Navier-Stokes equation with $\mathbb{P}_1$-iso-$\mathbb{P}_2/\mathbb{P}_1$ element is well known (see \cite{Temam1978,He2008}). More references can be found in \cite{AlexandreErn2004,Girault1986,John2016}. To solve the time continuous FE approximation problem, we consider the Implicit/Explicit (IMEX) Euler scheme  (see \cite{He2008}). However, other time schemes can be employed, e.g. implicit Euler and implicit Crank-Nicholson schemes \cite{Heywood1990,John2016}. In the multi-level setup, we recursively bisect the time discretization at level $l$ to get time discretization at level $l+1$. For $n = 1, 2, ..., 2^l$, we define $t_n = n \Delta t_l$, where $\Delta t_l  = T 2^{-l}$. We denote by $u^{I,l,n}, p^{I,l,n}$ the solution at time step $t_n$. Hence under the implicit/explicit Euler scheme, we solve the following saddle point problem for each time step. 
\begin{equation}
\label{chap8:imex discretized eqn}
\begin{cases}
    \text{Seek } u^{I,l,n+1} \in V^l \text{ and } p^{I,l,n+1} \in M^l \text{ such that} \\
    (\frac{u^{I,l,n+1}}{\Delta t}, v^l) + (\nu \nabla u^{I,l,n+1}, \nabla v^l) + ((u^{I,l,n} \cdot \nabla) u^{I,l, n+1}, v^l)  \\
    \qquad\qquad\qquad\qquad\qquad\qquad\qquad\quad  + (\nabla p^{I,l,n+1}, v^l) = (f^J + \frac{u^{I,l,n}}{\Delta t}, v^l), &\forall v^l \in V^l, \\
    (\nabla \cdot u^{I,l,n+1}, r^l) = 0,  &\forall r^l \in M^l. 
\end{cases}    
\end{equation}
Let the operator $a(u^{I,l,n}; u^{I,l,n+1}, v^l) = (\frac{u^{I,l,n+1}}{\Delta t}, v^l) + (\nu \nabla u^{I,l,n+1}, \nabla v^l) + ((u^{I,l,n} \cdot \nabla) u^{I,l, n+1}, v^l)$ and operator $b(u^{I,l,n+1}, r^l) =- (\nabla \cdot u^{I,l,n+1}, r^l)$. We have the following saddle point problem, 
\begin{equation}
\label{chap8:discrete saddle point problem}
\begin{cases}
    \text{Seek } u^{I,l,n+1} \in V^l \text{ and } p^{I,l,n+1} \in M^l \text{ such that} \\
    a(u^{I,l,n}; u^{I,l,n+1}, v^l) + b(v^l, p^{I,l,n+1}) = (f^J + \frac{u^{I,l,n}}{\Delta t}, v^l), &\forall v^l \in V^l \\
    b(u^{I,l,n+1}, r^l) = 0,  &\forall r^l \in M^l. 
\end{cases}    
\end{equation}
To solve Stokes like equations, FE spaces have to satisfy the inf-sup condition uniformly with respect to $l$. Ern and Guermond \cite{AlexandreErn2004} provide a comprehensive list of FE spaces that satisfy the inf-sup condition. Bubble element, Taylor-Hood element $(\mathbb{P}_2 /\mathbb{P}_1 )$ and $\mathbb{P}_1$-iso-$\mathbb{P}_2 /\mathbb{P}_1$ element are some commonly known elements that satisfy the inf-sup conditions and widely adopted for solving fluid mechanics problems. More detailed error analysis of the implicit/explicit Euler scheme can be found in He’s paper \cite{He2008}; and analysis for second order time scheme can be found in \cite{Heywood1990}. We choose the implicit/explicit Euler scheme in our analysis. To further estimate the FE approximation error with the implicit/explicit Euler scheme, we make the following regularity assumption. 
\begin{assumption}
\label{chap8:regularity assumption on u0 and f}
We assume $\bar{u} _0\in D(A)$,  $\Bar{f}, \Bar{f}_t, \Bar{f}_{tt}\in L^{\infty}(0, T; [L^2(\mathbb{T})]^2)$ , and $ \phi_i \in D(A)$ and $\psi^i, \psi^i_{t}, \psi^i_{tt} \in L^{\infty}(0, T; [L^2(\mathbb{T})]^2)$ for all $i \in \mathbb{N}$ such that
\begin{equation}
\|A \bar{u} _0\|_{L^2(D)} + \sum_{i\geq 1} \|A \phi_i\|_{L^2(D)} \leq C,
\end{equation}
and
\begin{equation}
\begin{aligned}
\sup_{0 \leq t \leq T} \{\|\Bar{f}(t)&\|_{L^2(\mathbb{T}^2)} + \|\Bar{f}_t(t)\|_{L^2(\mathbb{T}^2)}+\|\Bar{f}_{tt}(t)\|_{L^2(\mathbb{T}^2)}\} \\
&+ \sup_{0 \leq t \leq T} \{\sum_{i\geq 1 }(\|\psi^i(t)\|_{L^2(\mathbb{T}^2)}+ \|\psi^i_t(t)\|_{L^2(\mathbb{T}^2)}+\|\psi^i_{tt}(t)\|_{L^2(\mathbb{T}^2)})\} \leq C,
\end{aligned}
\end{equation}
where $C$ is a positive constant and $A=-P \Delta$ is the Stokes operator; $P$ is the orthogonal projection from $L^2 \rightarrow \mathbf{H}$.
\end{assumption}
With assumption \ref{chap8:regularity assumption on u0 and f}, we deduce
\begin{equation}
\|A u_0\|_{L^2(\mathbb{T}^2)} + \sup_{0 \leq t \leq T} \{\|f(t)\|_{L^2(\mathbb{T}^2)}+ \|f_t(t)\|_{L^2(\mathbb{T}^2)}+\|f_{tt}(t)\|_{L^2(\mathbb{T}^2)}\} \leq C. 
\end{equation}

With the regularity assumption, we have the following error estimate for the fully discretized Navier-Stokes equation with Euler implicit/explict method from He \cite{He2008}. There are two positive constants $K_1$ and $K_2$ such that,
\begin{equation}
\label{chap8:IMEX approximation}
\begin{aligned}
&\left\|u^J(t_n)-u^{J,l,n}\right\|_{V} \leq K_{1}\left(h_l+\Delta t_l\right), \\
&\left\|p^J(t_n)-p^{J,l,n}\right\|_{L^{2}(\mathbb{T}^2)} \leq K_{2}\left(h_l+\Delta t_l\right),
\end{aligned}
\end{equation}
where $h_l = O(2^{-l})$ and $\Delta t_l = O(2^{-l})$ for all $n=1, ..., 2^l$. In particular $K_1$ and $K_2$ are independent from $\zeta$ and $\xi$. Hence we have the following error bound.
\begin{proposition}
\label{chap8:forward approximation}
Consider the FE approximation of the truncated mixed problem, with the error estimate (\ref{chap8:IMEX approximation}). There is a constant $C > 0$ such that for every $J, l \in \mathbb{N}$ and for every $(\zeta,\xi) \in U$, the following error bound holds 
\begin{equation}
\begin{aligned}
\|u(t_n, \cdot,\zeta,\xi)-u^{J,l,n}(\cdot, \zeta, \xi)\|_V &\leq C (2^{-l}+J^{-q}).
\end{aligned}
\end{equation}
\end{proposition}
\begin{proof}
The proposition can be proved by considering the truncation error and the FE approximation error,
\begin{align*}
\|u(t_n, \cdot,\zeta,\xi)&-u^{J,l,n}(\cdot, \zeta, \xi)\|_V \\
& \leq \|u(t_n, \cdot,\zeta,\xi)-u^{J}(t_n,\cdot, \zeta, \xi)\|_V + \|u^{J}(t_n,\cdot,\zeta,\xi)-u^{J,l,n}(\cdot, \zeta, \xi)\|_V \\
&\leq C(2^{-l}+J^{-q}).
\end{align*}
\end{proof}
To solve the ODE equation of the tracer trajectories, we use the backward Euler method, which leads to the following discretized equation,
\begin{equation}
\label{eqn:euler_backward}
    z_j^{n+1} = z_j^n + \Delta t \cdot u^{I,l,n+1}(z_j^n),
\end{equation}
where $z_j^0 = z_j(t=0)$ and $\Delta t$ is the time step size. 
\begin{lemma}
\begin{equation}
    |z_j(t_n) - z_j^n| \leq C(t_n) \cdot \Delta t
\end{equation}
\end{lemma}
\begin{proof}
The euler method is given as in \eqref{eqn:euler_backward}. We consider the following Taylor expansion,
\begin{equation*}
z_j(t_{n+1}) = z_j(t_n) + \Delta t \cdot z_j'(t_n) + (\Delta t)^2 \cdot z_j''(t_n) + ... .
\end{equation*}
where we have $z_j''(t_n) = \frac{\partial u}{\partial t}(t_n, z_j(t_n)) + \frac{\partial u}{\partial z_j}(t_n, z_j(t_n))u(t_n, z_j(t_n))$ hence $|z_j''(t_n)|$ is bounded. Now we have the Taylor approximation,
\begin{equation}
\label{eqn:taylor approximation}
z_j(t_{n+1}) = z_j(t_n) + \Delta t \cdot u(t_n, z_j(t_n)) + \mathcal{O}((\Delta t)^2).
\end{equation}
Let $e_{n} = z_j(t_{n+1}) - z_j^{n+1}$. By comparing equation \eqref{eqn:euler_backward} and equation \eqref{eqn:taylor approximation}, we have
\begin{align*}
e_{n+1} &= e_n + \Delta t \cdot (u_h(t_n, z_j^n) - u(t_n, z_j(t_n))) + \mathcal{O}((\Delta t)^2) \\
&= e_n + \Delta t \cdot ( u_h(t_n, z_j^n) - u(t_n, z_j^n) + u(t_n, z_j^n) - u(t_n, z_j(t_n))  ) + \mathcal{O}((\Delta t)^2) \\
&= e_n + \Delta t \cdot ( u_h(t_n, z_n) - u(t_n, z_n)) + \Delta t \cdot (u(t_n, z(t_n)+e_n) - u(t_n, z(t_n))  ) + \mathcal{O}((\Delta t)^2) \\
& \leq (1+\Delta t \cdot  \lambda)|e_{n}| + \Delta t \cdot | u_h(t_n, z_n) - u(t_n, z_n)|_{L_{\infty}} + c(\Delta t)^2 \\
& \leq (1+\Delta t \cdot  \lambda)|e_{n}| + \Delta t \cdot | u_h(t_n, z_n) - u(t_n, z_n)|_{1+s} + c(\Delta t)^2 
\end{align*}
From this we can have the following,
\begin{align*}
|e_{n+1}| \leq (1+\Delta t \cdot  \lambda)|e_{n}| + c(\Delta t)^2.
\end{align*}
with $h$ being linearly proportional to $\Delta t$. From deduction, we can get 
\begin{equation}
    |e_n| \leq \frac{c}{\lambda} h ((1+h\lambda)^n-1) , 
\end{equation}
and subsquently we have
\begin{equation}
    |e_n| \leq \frac{c}{\lambda}(e^{\lambda t_{n}}-1)\cdot \Delta t. 
\end{equation}
\end{proof}

\section{Multilevel MCMC}
We follow the Multilevel Markov Chain Monte Carlo method first developed in \cite{Hoang2013} for elliptic equation and subsequently applied to the Navier Stokes equations with Eulerian observations in \cite{MR4523340}. We denote the posterior expectation of $\ell(u(t, \cdot, \zeta, \xi))$, where $\ell$, where $\ell$ is a bounded linear map on $V$. Similar to the setup in the above mentioned work, we choose $J$ s.t. $J^{-q} = \mathcal{O}(2^{-l})$ and denote $u^{J_L, l}$ as $u^{l}$, $\theta^{J_l, l}$ as $\gamma^{J_l, l, \delta}$ as $\gamma^l$. We also have the hierarchy of approximations $\gamma^l$ of the posterior measure $\gamma^{\delta}$. The MLMCMC estiamtor $E_L^{MLMCMC}[\ell(u]$ of $\mathbb{E}[\ell(u]$ is  

\begin{equation}
\label{chap4:MLMCMC estimator}
\begin{aligned}
    &E_L^{MLMCMC}[\ell(u)] \\
    &= \sum_{l=1}^L\sum_{l'=1}^{L'(l)} E_{M_{ll'}}^{\gamma^{l}} [(1-\exp(\Phi^{l}(\cdot,\cdot; \delta)-\Phi^{l-1}(\cdot,\cdot; \delta)))(\ell(u^{l'})-\ell(u^{l'-1}))]\\
    &+ \sum_{l=1}^L\sum_{l'=1}^{L'(l)} E_{M_{ll'}}^{\gamma^{l}}[\exp(\Phi^{l}(\cdot,\cdot;\delta)-\Phi^{l-1}(\cdot,\cdot;\delta))-1]\cdot E_{M_{ll'}}^{\gamma^{l-1}}[\ell(u^{l'})-\ell(u^{l'-1})] \\
    &+ \sum_{l=1}^L E_{M_{l0}}^{\gamma^{l}}[(1-\exp(\Phi^l(\cdot,\cdot; \delta)-\Phi^{l-1}(\cdot,\cdot; \delta)))(\ell(u^0))] \\
    &+ \sum_{l=1}^L E_{M_{l0}}^{\gamma^{l}}[\exp(\Phi^l(\cdot,\cdot;\delta)-\Phi^{l-1}(\cdot,\cdot;\delta))-1]\cdot E_{M_{l0}}^{\gamma^{l-1}}[\ell(u^0)] \\
    &+ \sum_{l'=1}^{L'(0)} E_{M_{0l'}}^{\gamma^{0}}[\ell(u^{l'})-\ell(u^{l'-1})] + E_{M_{00}}^{\gamma^{0}}[\ell(u^{0})]
\end{aligned}
\end{equation}
where $L'(l)$ is chosen below; and $E_{M_{ll'}}^{\gamma^l}$ denotes the MCMC sample average of the Markov chain generated by MCMC sampling procedure with the acceptance probability
\[
\alpha^{l}((\zeta,\xi), (\zeta',\xi')) = 1 \wedge \exp(\Phi^{l}(\zeta,\xi;\delta)-\Phi^{l}(\zeta',\xi'; \delta)), \quad (\zeta,\xi),(\zeta',\xi')  \in U,
\]
for the independence sampler and the pCN sampler we employ in the numerical implementation, with $M_{ll'}$ samples. 

With the following multi-level sampling numbers,
\begin{equation}
    L'(l) := L-l, \quad \text{ and } \quad M_{ll'} := 2^{2(L-(l+l'))},
\end{equation}
we have the following error estimate,
\begin{equation}
\label{chap8:MLMCMC error estimate}
    \mathbf{E}[|\mathbb{E}^{\gamma^{\delta}}[\ell(u)]-{E}_L^{MLMCMC}[\ell(u)]|]\leq C(\delta)L^2 2^{-L},
\end{equation}
where $\mathbf E$ denotes the expectation in the probability space of all the Markov chains in the MLMCMC sampling procedure. 
To reduce the effect of the $L^2$ multiplying factor in \eqref{chap8:MLMCMC error estimate}, we can slightly enlarge the sample size $M_{ll'}$ as
\begin{equation}
M_{ll'}=(l+l')^a2^{2(L-(l+l'))}
\label{eq:a}
\end{equation}
for $a>0$.
We have the following result from \cite{Hoang2020}.
\begin{table}[htbp!]
    \centering
    \caption{Total MLMCMC error with different sample size choices for uniform prior}
    \begin{tabular}{|c|c|c|c|c|}
    \hline$a$ & $M_{l l^{\prime}}, l, l^{\prime}>1$ & $M_{l 0}=M_{0 l}$ & $M_{00}$ & Total error \\
    \hline 0 & $2^{2\left(L-\left(l+l'\right)\right)}$ & $2^{2(L-l)} / L^{2}$ & $2^{2 L} / L^{4}$ & $O\left(L^{2} 2^{-L}\right)$ \\
    \hline 2 & $\left(l+l^{\prime}\right)^{2} 2^{2\left(L-\left(l+l^{\prime}\right)\right)}$ & $2^{2(L-l)}$ & $2^{2 L} / L^{2}$ & $O\left(L \log L 2^{-L}\right)$ \\
    \hline 3 & $\left(l+l^{\prime}\right)^{3} 2^{2\left(L-\left(l+l^{\prime}\right)\right)}$ & $l 2^{2(L-l)}$ & $2^{2 L} / L$ & $O\left(L^{1 / 2} 2^{-L}\right)$ \\
    \hline 4 & $\left(l+l^{\prime}\right)^{4} 2^{2\left(L-\left(l+l^{\prime}\right)\right)}$ & $l^{2} 2^{2(L-l)}$ & $2^{2 L} /\left(\log L^{2}\right)$ & $O\left(\log L 2^{-L}\right)$ \\
    \hline
    \end{tabular}
    \label{chap8:numerical results}
\end{table}
%

{The  number of degrees of freedom for each time step is $O(2^{2l})$ for the FE resolution mesh size $O(2^{-l})$. For $O(2^l)$ time steps, the total number of degrees of freedom required for the forward solver  is  $O(2^{3l})$. Thus the total number of degrees of freedom required for computing the MLMCMC estimator with the finest resolution level $O(2^{-L})$ is
\begin{equation}
\lesssim \sum_{l=1}^L\sum_{l'=1}^{L-l} (l+l')^a2^{2(L-(l+l'))}(2^{3l}+2^{3l'})+\sum_{l=1}^L l^a2^{2(L-l)}2^{3l}+\sum_{l'=1}^L {l'}^a2^{3l'}+2^{2L}
\lesssim L^a2^{3L}.
\label{eq:dofs}
\end{equation}

\section{Numerical experiments}
In this section we implement the FE-MLMCMC method for approximating quantities of interest for Bayesian inverse problems for Navier-Stokes equation analyzed in the preceding sections. We employ the $\mathbb{Q}_1$-iso-$
\mathbb{Q}_2/\mathbb{Q}_1$ elements and Euler implicit/explicit time scheme for solving the forward Navier-Stokes equation. By lemma 4.27 of \cite{AlexandreErn2004}, the $\mathbb{Q}_1$-iso-$\mathbb{Q}_2$/$\mathbb{Q}_1$ elements satisfy the inf-sup condition. At each time step we  solve a saddle point system. We use the iterative FGMRES method with Schur complement precondition to solve the linear system efficiently. To illustrate the theoretical result, we consider the case where the forcing depends on one random variable as in this case we can compute a reference posterior expectation highly accurately, though we stress that the numerical method works for the case where the forcing and the initial condition depend on many random variables as shown theoretically above. 
In the first experiment, we consider the inverse problem of Navier-Stokes equation in the $(0, 1) \times (0, 1)$ square domain with the periodic boundary condition. For $\xi \sim U(0, 1.0)$, we consider the following model problem
\begin{equation}
\label{numerical problem setup}
    \begin{cases}
    \frac{\partial u}{\partial t} + u \cdot \nabla u - \nu \Delta u + \nabla p = f, &\quad \text{for } x\in \mathbb{T}^2, \\
    \nabla \cdot u = 0, &\quad \text{for } x\in\mathbb{T}^2, \\
    u(x,0) = 0;
    \end{cases}
\end{equation}
with the periodic boundary condition and the random forcing
\begin{equation}
    f = \begin{cases}
    \xi (\cos(2 \pi x_1)\sin(2 \pi x_2)+1.0) \exp(t),    \\
    -\xi (\sin(2 \pi x_1) \cos(2 \pi x_2)+1.0) \exp(t).
\end{cases}
\end{equation}
The forward observation functional is $\{z_i(t=0.5)\}_{i=1}^5, \{z_i(t=1.0)\}_{i=1}^5$. The quantity of interest is 
\begin{equation}
    100 \int_D \left(x_1^{1/2}x_2^{1/2}\frac{\partial u_1(x, 1.0)}{\partial x_2} - x_1^{1/2}x_2^{1/2}\frac{\partial u_2 (x, 1.0)}{\partial x_1}\right)dx.
\end{equation}
With the Gaussian prior, we generate a random realization of the solution by solving the forward problem with a randomly generated $\xi$. Then a random observation is obtained with a randomly generated Gaussian noise from $N(0, 1.0)$ to the forward functional. The observed tracer location is $z_1=\{(0.640, 0.795), (0.172, 0.195), \\ (1.134, 0.079), (0.173, -0.027), (0.846, 0.411)\}$ and $z_2 = \{(1.112, 0.323), (0.659, -0.303), (1.628, -0.429), \\ (0.656, -0.524), (1.353, -0.089)\}$. The initial tracer locations are $\{(0.491,0.908), (0.012,0.312), (0.975,0.203), \\ (0.035,0.111), (0.725,0.575)\}$. 
The reference posterior expectation of the quantity of interest is $-1.0924$ which is  computed by Gauss-Legendre quadrature with the highly accurate Fourier spectral forward solver with a total of $128 \times 128$ spectral points and 4th order accurate Runge-Kutta explicit time stepping with 0.0001 time step size. Detailed implementation of the Fourier spectral method for Navier-Stokes equation can be found in \cite{Peyret2002}. 

\begin{table}[]
    \centering
    \begin{tabular}{|c|c|c|c|c|c|c|} \hline
        L &  1 & 2 & 3 & 4 & 5 & 6\\ \hline
        error (a=2) & 0.663275 & 0.338971 & 0.252836 & 0.1122 & 0.062769 & 0.031369\\ \hline
        error (a=3) & 0.663275 & 0.397023 & 0.190823 & 0.078067 & 0.032926 & \\ \hline
    \end{tabular}
    \caption{Error Table}
    \label{tab:my_label}
\end{table}

\begin{figure}[htbp!]
    \centering
    \caption{MLMCMC error for 2D Navier-Stokes equation with uniform prior, a=2}
    \includegraphics[width=0.8\textwidth]{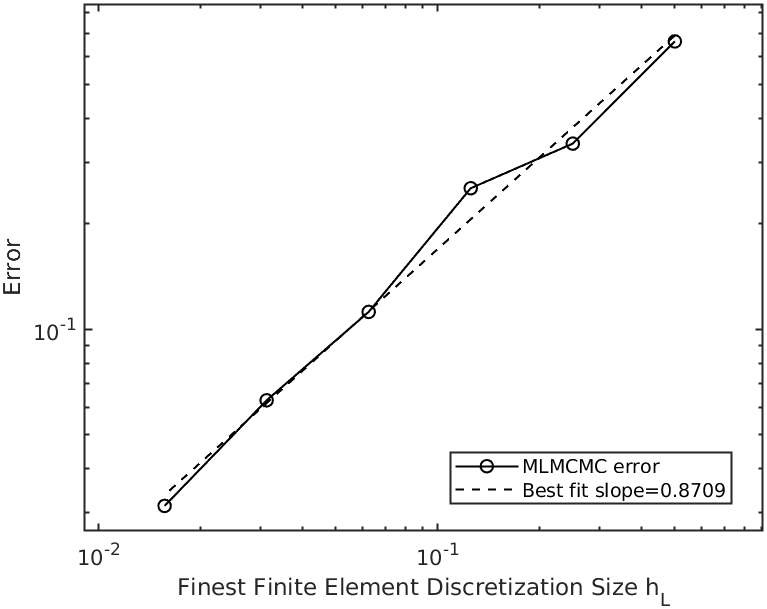}
\end{figure}

\begin{figure}[htbp!]
    \centering
    \caption{MLMCMC error for 2D Navier-Stokes equation with uniform prior, a=3}
    \includegraphics[width=0.8\textwidth]{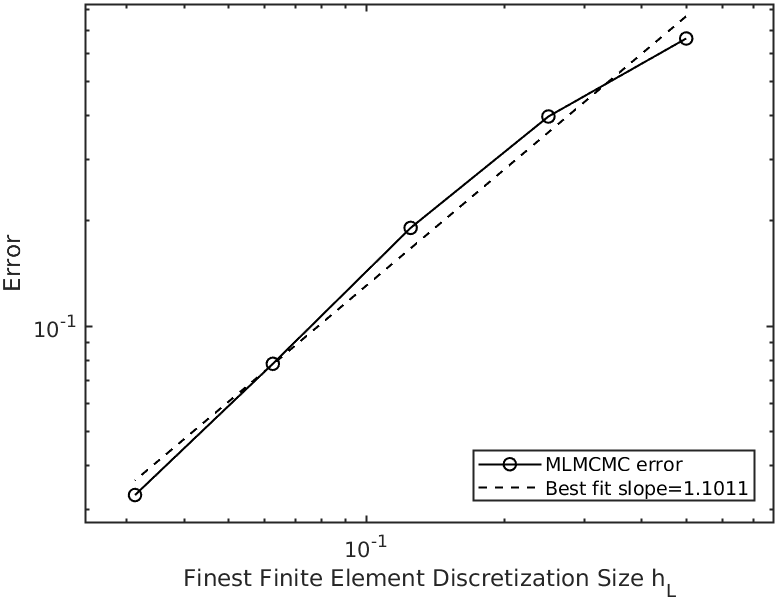}
\end{figure}

\bibliographystyle{siam}
\bibliography{mybibfile}

\end{document}